\documentclass[a4paper,12pt]{article}
\usepackage[utf8]{inputenc}
\usepackage{amsmath,amsthm,amsfonts,latexsym,amssymb,bm,cite}
\usepackage{cite}

\newtheorem{Thm}{Theorem}[section]

\newtheorem{Lem}[Thm]{Lemma}

\newtheorem{Prop}[Thm]{Proposition}
\newtheorem{Claim}[Thm]{Claim}

\newtheorem{Def}[Thm]{Definition}

\newenvironment{altproof}[1]   
{\noindent
{\bf Proof of {#1}}.}   
{\nopagebreak\mbox{}\hfill $\Box$\par\addvspace{0.5cm}}

\newcommand{\R}{\mathbb{R}}
\newcommand{\ut}{\tilde{u}} 
\newcommand{\utp}{\tilde{u}^+}
\newcommand{\utm}{\tilde{u}^-}
\newcommand{\uh}{\hat{u}}
\newcommand{\uhp}{\hat{u}^+}
\newcommand{\uhm}{\hat{u}^-}
\newcommand{\ub}{\bar{u}}
\newcommand{\bu}{\underline{u}}
 
\newcommand{\vtp}{\tilde{v}^+}
\newcommand{\vtm}{\tilde{v}^-}

\newcommand{\vhp}{\hat{v}^+}

\newcommand{\utpn}{\tilde{u}^+_n}
\newcommand{\utmn}{\tilde{u}^-_n}

\newcommand{\uhpn}{\hat{u}^+_n}
\newcommand{\uhmn}{\hat{u}^-_n}
\newcommand{\ubn}{\bar{u}_n}
\newcommand{\bun}{\underline{u}_n}

\newcommand{\ot}{\tilde{\omega}}
\newcommand{\ohat}{\hat{\omega}}
\newcommand{\ob}{\bar{\omega}}
\newcommand{\rt}{\tilde{r}}
\newcommand{\st}{\tilde{s}}
\newcommand{\that}{\hat{t}}

\newcommand{\hozo}{H^1_0(\Omega)}
\newcommand{\hoz}{\underline{H}}
\newcommand{\iu}{I_\mu}
\newcommand{\ap}{a^+}
\newcommand{\am}{a^-}
\newcommand{\nmu}{{\cal{N}}_\mu}

\newcommand{\op}{\frac{1}{p}}
\newcommand{\oh}{\frac{1}{2}}
\newcommand{\ott}{\frac{1}{\theta}}

\newcommand{\mm}{\cal{M}}

\newcommand{\nutp}{\left(\int\ap(\utp)^2\right)^{1/2}}
\newcommand{\nutm}{\left(\int\ap(\utm)^2\right)^{1/2}}
\newcommand{\nuhp}{\left(\int\ap(\uhp)^2\right)^{1/2}}

\newcommand{\jtp}{\tilde{J}_\mu^+}
\newcommand{\jtm}{\tilde{J}_\mu^-}
\newcommand{\jhp}{\hat{J}_\mu^+}
\newcommand{\g}{\check{f}}

\newcommand{\x}{\, \cdot\, ,}
\newcommand{\xx}{x,}
\newcommand{\xxx}{\, \cdot\, }
\newcommand{\str}{\underline{S}}
\newcommand{\ztp}{\tilde{z}^+}
\newcommand{\ztm}{\tilde{z}^-}
\newcommand{\zhp}{\hat{z}^+}
\newcommand{\Lt}{\Lambda}

\begin{document}
\begin{center}
{\Large\bf 
Multibump nodal solutions for an indefinite nonhomogeneous elliptic problem}\\ 
\ \\
Pedro M.\ Girão\footnote{Email: pgirao@math.ist.utl.pt. Partially supported by the Center for Mathematical
Analysis, Geometry and Dynamical Systems through FCT Program POCTI/FEDER.} and José Maria Gomes\footnote{Email:
jgomes@math.ist.utl.pt. Supported by FCT
grant SFRH/BPD/29098/2006.}
\\
Instituto Superior Técnico\\
Av.\ Rovisco Pais\\
1049-001 Lisbon, Portugal
\end{center}

\begin{center}
{\bf Abstract}\\
\end{center}
\noindent We construct multibump nodal solutions of the elliptic equation
$$
-\Delta u=\ap[\lambda u+ f(\x u)]-\mu\am g(\x u)
$$
in $\hozo$, when $\mu$ is large, under appropriate assumptions,
for $f$ superlinear and subcritical and such that
the eigenvalues of the associated linearized operator on $H^1_0(\{x\in\Omega:\: a(x)>0\})$
at zero,
$u\longmapsto u-\lambda(-\Delta)^{-1}(\ap u)$, 
are positive.
The solutions are of least energy in some Nehari-type set
defined by imposing suitable conditions on orthogonal
components of functions in $\hozo$.

\section{Introduction}

We are concerned with multibump solutions of the semilinear
Dirichlet problem
\begin{equation}
\label{zero}
\left\{\begin{array}{ll}
-\Delta u=\ap[\lambda u+ f(\x u)]-\mu\am g(\x u)&\mbox{in}\ \Omega,\\
u=0&\mbox{on}\ \partial\Omega.
\end{array}\right.
\end{equation}
We state our assumptions. The set $\Omega$ is an
open and bounded Lipschitz domain in $\R^N$, $N\geq 1$.
The function $a$ belongs to $C(\Omega)\cap L^\infty(\Omega)$ and
$\ap$ denotes $\max\{a,0\}$, $\am=\ap-a$ as usual.
The set 
$$
\Omega^+:=\{x\in\Omega:\: a(x)>0\}
$$
has, say, three components,
$$
\Omega^+=\ot\cup\ohat\cup\ob,
$$
also Lipschitz, and
$$
\Omega^-:=\{x\in\Omega:\: a(x)<0\}=\Omega\setminus\overline{\Omega^+}.
$$
The value $\mu$ is a nonnegative parameter.
Let $p$ be a superquadratic and subcritical exponent,
$2<p<2^*$, where $2^*=+\infty$ if $N=1$ or $N=2$,
and $2^*=2N/(N-2)$ if $N\geq 3$.
The functions $f$ and $g$, defined on $\Omega\times\R$, satisfy $f\in C(\Omega^+\times\R)$, $g\in C(\Omega^-\times\R)$,
$f$ is differentiable with respect to the second variable $u$ in $\Omega^+\times\R$ and
$f^\prime:=\partial f/\partial u\in C(\Omega^+\times\R)$.
Furthermore,
denoting by $F(\x u):=\int_0^uf(\x s)\,ds$, and $G(\x u):=\int_0^ug(\x s)\,ds$,
the functions $f$ and $g$ satisfy the following hypotheses:
\begin{enumerate}
 \item[(a)] $\displaystyle\exists_{C_0>0}\ \forall_{u}\ \ |f^\prime(\x u)|\leq C_0(1+|u|^{p-2})$,\\
$\displaystyle\exists_{C_0'>0}\ \forall_{u}\ \ |g(\x u)|\,\,\leq C_0'(1+|u|^{p-1})$.
 \item[(b)] $\displaystyle\exists_{\theta>2}\ \forall_{u\neq 0}\ \ 0<\theta F(\x u)\leq uf(\x u)$,\\
$\displaystyle\exists_{\vartheta>1}\ \forall_{u\neq 0}\ \ 0<\vartheta G(\x u)\leq ug(\x u)$.
 \item[(c)] $\displaystyle\forall_{u\neq 0}\ \forall_{x\in\Omega^+}\ \ \frac{f(\xx u)}{u}<f^\prime(\xx u)$.
 \item[(d)] 
Let $\tilde{\lambda}_1$, $\hat{\lambda}_1$, $\bar{\lambda}_1$ be the
first eigenvalue of $-\Delta u=l\ap u$
on $H^1_0(\ot)$, $H^1_0(\ohat)$, $H^1_0(\ob)$, respectively.
The parameter $\lambda$ satisfies
$$
0\leq\lambda<\Lambda_1:=\min\left\{\tilde{\lambda}_1,\hat{\lambda}_1,\bar{\lambda}_1\right\}.
$$
\end{enumerate}

From (b) it follows that $f(\x 0)\equiv f'(\x 0)\equiv g(\x 0)\equiv 0$.
Hypothesis (d) is equivalent to saying 
the parameter $\lambda$ is nonnegative and smaller than 
the maximum eigenvalue of the map from
$H^1_0(\Omega^+)$ to $H^1_0(\Omega^+)$
defined by
$u\mapsto (-\Delta)^{-1}(\ap u)$;
here $(-\Delta)^{-1}$ denotes the
inverse of the Dirichlet Laplacian on $H^1_0(\Omega^+)$.
An example of functions $f$ and $g$ satisfying our assumptions are
\begin{eqnarray*}
f(\x u)&=&a_1(\xxx)|u|^{p_1-2}u+a_2(\xxx)|u|^{p_2-2}u,\\
g(\x u)&=&b_1(\xxx)|u|^{q_1-2}u\,+b_2(\xxx)|u|^{q_2-2}u,
\end{eqnarray*}
with $2<p_1, p_2<2^*$, $1<q_1, q_2< 2^*$, $a_1, a_2\in C(\Omega^+)\cap L^\infty(\Omega^+)$,
$b_1, b_2\in C(\Omega^-)\cap L^\infty(\Omega^-)$.
In fact, $p_1, p_2, q_1, q_2$ might even be continuous functions of 
the space variable, with $p_1, p_2$ bounded away from 2 and $2^*$, and
$q_1, q_2$ bounded away from 1 and $2^*$. Our results would still hold if
we were to impose less on the function $g$, namely that it satisfied 
the inequality in (a) and $G(\x u)\geq c|u|^\vartheta$ for some $c>0$ and $\vartheta>1$.

We consider the usual inner product
$\left\langle u ,v \right\rangle=\int\nabla u\cdot\nabla v$ in $\hozo$,
and denote by $\left\|\ \right\|$ the induced norm. 
The differential equation in (\ref{zero}) is the Euler-Lagrange equation for
the energy functional $\iu:\hozo\to\R$,
\begin{equation}\label{EL}
\iu(u)=\oh\|u\|^2-\frac{\lambda}{2}\int_\Omega\ap u^2-\int_\Omega\left[\ap F(\x u) -\mu\am G(\x u)\right].
\end{equation}
Our main result is
\begin{Thm}\label{thm}
There exists $\check{\mu}$ such that for $\mu>\check{\mu}$ the equation
\begin{equation}
\label{one}
-\Delta u=\ap [\lambda u+f(\x u)]-\mu\am g(\x u)
\end{equation}
has an $\hozo$ weak solution $u_\mu$ and, when $\mu_n\to+\infty$, modulo a subsequence,
\begin{equation}
\label{fifteen}
u_{\mu_n}\to u\quad\quad\mbox{in}\ \hozo,
\end{equation}
where $u|_{\ot}$ is a least energy nodal solution of~{\rm (\ref{one})} in $H^1_0(\ot)$,
$u|_{\ohat}$ is a least energy positive solution of~{\rm (\ref{one})} in $H^1_0(\ohat)$,
and $u|_{\ob}$ and $u|_{\Omega^-}$ are zero. 
\end{Thm}

We assume for simplicity that the set $\Omega^+$ has three components.
But when $\Omega^+$ has a different number of components,
Theorem~\ref{thm} can be generalized in a way parallel to the one in \cite{GG}. 
In simple terms we may say that, when $\mu$ is large, one can choose
the solution to be positive, negative, nodal or vanish
in any given component of $\Omega^+$.

Theorem~\ref{thm} generalizes Proposition~2.1 of \cite{GG},
which addresses the case where $\lambda=0$, $f$ is homogeneous and $g=f$, more precisely,
$f(\x u)=g(\x u)=|u|^{p-2}u$ and
$2<p<2^*$. Even in this special case we improve our previous results.
Also, here all proofs are direct, no argument is by contradiction,
so that keeping track of the constants it is possible to give
an upper bound for~$\check{\mu}$.

We allow for a rather general situation for the nonlinearity.
Indefinite weights have also been considered in several other works.
The paper \cite{AT} concerns existence and multiplicity of positive solutions for elliptic equations whose nonlinear term has the form
$W(x)f(u)$ where $W$ changes sign. The paper \cite{AdP} studies equations with
an indefinite nonlinearity both using min-max methods and using Morse theory.
In particular, \cite{AdP} and \cite{AT} treat the delicate issue of conditions
on the indefinite weight and the nonlinearity that lead to the Palais-Smale condition.

The main ideas for the proof of Theorem~\ref{thm} are from 
\cite{BGH}, \cite{CCN}, \cite{GG} and~\cite{RT}.
More specifically,
existence of a sign-changing solution for a superlinear
problem was proved in \cite{CCN} by minimizing the Euler-Lagrange functional
over a Nehari-type set. The nonlinearity considered in
\cite{CCN} satisfied conditions similar to the ones
imposed to our function $f$. Using quite a different approach
to ours but an orthogonal decomposition of $\hozo$, \cite{BGH}
was the first work to establish the existence of multibump {\em positive}\/ solutions 
to~(\ref{zero}), for $N>1$, when $f$ and $g$ are equal and are 
homogeneous superlinear functions. The work \cite{RT}
used cut-off operators and minimization over a Nehari-type manifold
to construct positive multispike solutions for an elliptic system. 
The method of \cite{RT} and
the orthogonal decomposition of \cite{BGH} suggested 
the variational framework in \cite{GG}, used to prove the
existence of multibump nodal solutions to (\ref{zero})
in the special case for $f$ and $g$ mentioned above.
The technique from \cite{GG} is the one we explore here.
We would like to emphasize that our solutions are of
least energy in a set $\nmu$ which is not a manifold
(see \cite[Lemmas~3.1 and 3.2]{BW}). In fact,
in the present case not even $\nmu\cap H^2(\Omega)$ is
a manifold, although it does admit a tangent space
at the minimum $u_\mu$.

The earliest successes in gluing mountain pass solutions
of nonlinear elliptic equations and Hamiltonian systems
came from \cite{CZR1}, \cite{CZR2} and \cite{S}.
The process was simplified by using an alternative 
procedure in \cite{LW}, which allowed the authors
to glue minimizers on the Nehari manifold together as
genuine solutions.

Related local Nehari manifold approaches have already
been used in other problems. In \cite{MP} a technique which
resembles the one in this paper leads to multibump solutions
of a semilinear elliptic Dirichlet problem with an operator in
divergence form. The solutions are 
associated to distinct vanishing components of an
asymptotically vanishing coefficient. If the degeneration set 
consists of $k$ connected components, then existence of at least
$2^k-1$ distinct positive solutions, which concentrate on the degeneration set, is established.

It is also important to mention \cite{dPF}, a motivation of \cite{RT}.
Gluing through local Nehari manifolds was also used in \cite{dPF2}.

Recent interesting 
related results can be found in \cite{AW} and \cite{BCW}.

The organization of this paper is as follows.
In Section~\ref{section2} we define the Nehari-type set $\nmu$ and
give estimates for low energy functions.
In Section~\ref{section3} we prove existence of least energy
solutions in $\nmu$. We also characterize the strong limit of these
solutions as $\mu\to+\infty$. A couple of more technical proofs 
are left to the Appendix.

\section{A Nehari-type set and estimates for low energy functions}\label{section2}

\setcounter{equation}{0}
Let $\varpi$ be equal to $\ot$, $\ohat$ or $\ob$.
Because we assume $\varpi$ is Lipschitz, if 
$u\in\hozo$ and $u\equiv 0$ on the complement of $\varpi$,
then $u|_{\varpi}$ belongs to $H^1_0(\varpi)$.
We define
\begin{eqnarray*}
\hoz(\varpi)&=&\left\{u\in\hozo:\:u=0\ \mbox{in}\ \Omega\setminus\varpi\right\}.
\end{eqnarray*}
For $u$ in $\hozo$,
we denote by $\ut$, $\uh$, $\ub$ and $\bu$ the orthogonal projections of $u$ on 
the orthogonal spaces
$\hoz(\ot)$, $\hoz(\ohat)$, $\hoz(\ob)$ and 
$[\hoz(\ot)\oplus\hoz(\ohat)\oplus\hoz(\ob)]^\bot$. 
The function $\bu$ is harmonic in $\Omega^+$.

Clearly, the derivative of the energy functional $\iu$ in (\ref{EL}) is
$$
\iu'(u)(z)=\left\langle u ,z \right\rangle-\int\left[\ap [\lambda u+f(\x u)]z-\mu\am g(\x u)z\right],
$$
for $u, z\in\hozo$.
The solutions $u_\mu$ of (\ref{one}) in Theorem~\ref{thm} will be obtained by minimizing the functional
$\iu$ on a Nehari-type set which we will soon define. 
First we need some parameters.
We set
$\Lt=(\lambda/\Lambda_1+1)/2<1$
and 
\begin{equation}\label{e1}
\gamma=\frac{1-\Lt}{4}. 
\end{equation}
We denote by $c_p$ be the Sobolev constant
$$ 
c_p|u|_p\leq\|u\|,
$$ 
with $|u|_p=\left(\int u^p\right)^{1/p}$ the $L^p(\Omega)$ norm of $u$.
When the region of integration is not explicitly indicated,
it is understood that integrals are over $\Omega$.
Next we will obtain a lower bound for $\iu(u)$ when $u\in H^1_0(\Omega^+)$.
Consider the set
\begin{equation}\label{s}
S:=\left\{x\in\Omega^+:\:\mbox{dist}\,(x,\R^N\setminus\Omega^+)<{1}/{n_0}\right\},
\end{equation}
where $n_0$ is large enough so that 
$$
|S|\leq\textstyle\left(\frac{\gamma c_{2^*}^2}{\sup\ap (\lambda+C_0)}\right)^{2^*/(2^*-2)},
$$
with $C_0$ as in (a). 
Here and henceforth, when $N=1$ or 2 it should be understood that instead of $2^*$
a fixed exponent greater than $p$ should appear. 
For $u\in\hozo$,
\begin{eqnarray}
\left|\int_S\ap [\lambda u+f(\x u)u]\right|&\leq&(\lambda+C_0)\int_S\ap |u|^2
+C_0\int_S\ap|u|^p\nonumber\\
&\leq&\sup\ap (\lambda+C_0)|u|_{2^*}^2|S|^{(2^*-2)/2^*}+C_0\int_S\ap|u|^p\nonumber\\
&\leq&\frac{\sup\ap (\lambda+C_0)}{c_{2^*}^2}\|u\|^2|S|^{(2^*-2)/2^*}+C_0\int_S\ap|u|^p\nonumber\\
&\leq&\gamma\|u\|^2+C_0\int_S\ap|u|^p.\label{ap1}
\end{eqnarray}
There exists a constant $C_1$, which we may assume greater
than or equal to $C_0$, such that
\begin{equation}
\label{if}
\lambda u^2+f(\xx u)u\leq {\textstyle\frac{1}{2}}
(\lambda+\Lambda_1)u^2+C_1|u|^p\quad\mbox{for}\ u\in\R\ \mbox{and}\ x\in\Omega^+\setminus S.
\end{equation}
For 
$u\in H^1_0(\Omega^+)$.
\begin{eqnarray}
\int_{\Omega^+\setminus S}\ap [\lambda u+f(\x u)]u&\leq&
\int_{\Omega^+\setminus S}\ap
\left[{\textstyle\frac{1}{2}}(\lambda+\Lambda_1)u^2+C_1|u|^p\right]\nonumber\\
&\leq&\Lt\|u\|^2+C_1\int_{\Omega^+\setminus S}\ap|u|^p.\label{ap2}
\end{eqnarray}
Combining (\ref{ap1}) and (\ref{ap2}), we obtain 
\begin{equation}\label{ap3}
\int_{\Omega^+}\ap [\lambda u+f(\x u)]u\leq{\textstyle\frac{3\Lt+1}{4}}\|u\|^2+C_1\int_{\Omega^+}\ap|u|^p
\end{equation}
for $u\in H^1_0(\Omega^+)$. We let
$$
\rho=\textstyle\left(\frac{\gamma c_p^{p}}{\sup\ap C_1}\right)^{1/(p-2)}.
$$
We use inequality (\ref{ap3}) and (b) to obtain the lower bounds
\begin{eqnarray}
\iu(u)&\geq& \frac{1}{2}\|u\|^2-\frac{\lambda}{2}\int\ap u^2-\frac{1}{\theta}\int\ap f(\x u)u\nonumber\\
&\geq&\frac{1}{2}\left(\|u\|^2-\lambda\int\ap u^2-\int\ap f(\x u)u\right)\nonumber\\
&\geq&\frac{3\gamma}{2}\|u\|^2-\frac{C_1}{2}\int\ap|u|^p\nonumber\\
&\geq&\frac{3\gamma}{2}\|u\|^2-\frac{\sup\ap C_1}{2c_p^p}\|u\|^p\nonumber\\
&\geq&\gamma\|u\|^2\label{p1},
\end{eqnarray}
for $u\in H^1_0(\Omega^+)$ such that $\|u\|\leq\rho$.

As in \cite{CCN}, one can prove there exists a function
$v\in\hozo$ such that $v=\vtp-\vtm+\vhp$ with 
$\vtp$, $\vtm$ and $\vhp\not\equiv 0$ and
$$
\iu'(v)(\vtp)=\iu'(v)(\vtm)=\iu'(v)(\vhp)=0.
$$
Finally, we let $R$ satisfy
$$ 
\iu(v)<\textstyle\left(1-\frac{2}{\theta}\right)(1-\Lambda)R^2\quad\mbox{and}\quad R>\rho.
$$ 

We are ready to give the definition of the Nehari-type set $\nmu$.
\begin{Def}\label{L_three} $\nmu$ is the set of functions $u=\ut+\uh+\ub+\bu\in\hozo$
satisfying
\begin{enumerate}
\item[{\rm (}${\cal{N}}_{i}${\rm )}] $\iu^\prime(u)(\utp)=\iu^\prime(u)(\utm)=\iu^\prime(u)(\uhp)=0$,
\item[{\rm (}${\cal{N}}_{ii}${\rm )}] $\utp$, $\utm$, $\uhp\not\equiv 0$,
\item[{\rm (}${\cal{N}}_{iii}${\rm )}] $\iu(u)\leq\iu(v)+1$,
\item[{\rm (}${\cal{N}}_{iv}${\rm )}] $\left\|\ut+\uhp\right\|\leq R$,
\item[{\rm (}${\cal{N}}_{v}${\rm )}] $\max\{\left\|\uhm\right\|,\left\|\ub\right\|\}\leq\rho$,
\item[{\rm (}${\cal{N}}_{vi}${\rm )}] 
$\left(\int\ap\bu^2\right)^{1/2}\!\leq\gamma\min\left\{
\left(\int\ap(\utp)^2\right)^{1/2}\!\!,\left(\int\ap(\utm)^2\right)^{1/2}\!\!,\left(\int\ap(\uhp)^2\right)^{1/2}\right\},$
\item[{\rm (}${\cal{N}}_{vii}${\rm )}]
 $\left\|\bu\right\|\leq\min\{\left\|\utp\right\|,\left\|\utm\right\|,\left\|\uhp\right\|\}$.
\end{enumerate}
\end{Def}
Note that $\nmu\neq\emptyset$ as $v\in\nmu$.
The conditions (${\cal{N}}_{vi}$) and (${\cal{N}}_{vii}$) 
are crucial to prove lower bounds on the norms of some of the components
of the functions in $\nmu$.

The next lemma will allow us to write 
the integrals $\int\!\ap u^2$ and $\int\!\ap F(\x u)$, for large $\mu$, as a sum of integrals in terms
of the components of $u$ plus a small error.
\begin{Lem}\label{L_one}
For any positive $\delta$, there exists $\mu_\delta$ such that, for all $\mu>\mu_\delta$,
$${u\in\nmu}\quad\Rightarrow\quad|\bu|_p^p<\delta.$$
\end{Lem}
The proof of Lemma~\ref{L_one} is given in the Appendix.
In the remainder of this section we establish a number of lemmas
which will be used to prove that, for large $\mu$, the functional $\iu$ has a minimum
on $\nmu$ and to prove that, for large $\mu$, every minimizer of $\iu$ on $\nmu$ is a critical point
of $\iu$. The next lemma will be used (via Lemma~\ref{kappa2}) in
connection with (${\cal{N}}_{ii}$) and in connection with (${\cal{N}}_{vii}$):
\begin{Lem}\label{L_two}
There exists a positive constant $\kappa_1$ such that for all $\mu$,
$$ 
{u\in\nmu}\quad\Rightarrow\quad\min\left\{\|\utp\|,\|\utm\|,\|\uhp\|\right\}\geq\kappa_1.
$$ 
\end{Lem}
\begin{proof}
Denote by $w$ one of the three functions $\utp$, $-\utm$ or $\uhp$; let 
$\varpi$ be $\ot$ for the first and second choices for $w$, and
be $\ohat$ for the third choice for $w$.
Note that on the support of $w$, we have $u=w+\bu$. Let $S$ be as in (\ref{s}).
By (${\cal{N}}_{vii}$) and a computation similar to (\ref{ap1}),
\begin{eqnarray*}
\int_S\ap f(\x u)w&\leq&2\gamma\|w\|^2
+2^{p-2}\sup\ap C_1\int_S(|w|^p+|\bu|^{p-1}|w|),
\end{eqnarray*}
whereas by (\ref{if}) and (${\cal{N}}_{vi}$) 
\begin{eqnarray*}
\int_{\Omega\setminus S}\ap f(\x u)w&\leq&(\gamma+1)
{\textstyle\frac{1}{2}}(\Lambda_1-\lambda)\int\ap
w^2\\ 
&&\quad +2^{p-2}\sup\ap C_1\int_{\Omega\setminus S}(|w|^p+|\bu|^{p-1}|w|)\\
&\leq&(1-\Lt)(\gamma+1)\|w\|^2\\
&&\quad +2^{p-2}\sup\ap C_1\int_{\Omega\setminus S}(|w|^p+|\bu|^{p-1}|w|).
\end{eqnarray*}
Similarly, we have
$$
\lambda\int\ap uw\leq (2\Lambda-1)(\gamma+1)\|w\|^2.
$$
So, by (${\cal{N}}_{i}$), (\ref{e1}) and (${\cal{N}}_{vii}$), a simple
computation leads to
\begin{eqnarray*}
\|w\|^2=\lambda\int\ap uw+\int\ap f(\x u)w&\leq&{\textstyle\frac{\Lt+3}{4}}\|w\|^2
+2^{p-1}\sup\ap\frac{C_1}{c_p^{p}}\|w\|^p.
\end{eqnarray*}
From (${\cal{N}}_{ii}$), $\|w\|$ is bounded below by 
$$\kappa_1=\textstyle\frac{1}{2}\left(
\frac{\gamma c_p^{p}}{2\sup\ap C_1}\right)^{1/(p-2)}.$$
\end{proof}
The next lemma will be used in connection with (${\cal{N}}_{iv}$).
We denote by $o(1)$ a quantity whose absolute value can be made arbitrarily small,
uniformly in $u\in\nmu$, when $\mu$ is large.
\begin{Lem}\label{overr}
Let $\hat{R}>\overline{R}>0$. 
If $\mu$ is sufficiently large, then 
$$u\in\nmu\ \wedge\ 
\iu(u)\leq\left(1-\frac{2}{\theta}\right)(1-\Lambda)\overline{R}^2\quad\Rightarrow\quad 
\|\ut+\uhp\|\leq\hat{R}.
$$
\end{Lem}
\begin{proof}
Let $u\in\nmu$.
We bound from below $\iu(u)$ by an expression involving
the norms of the components of $u$.
From (b) and (${\cal{N}}_{i}$),
\begin{eqnarray}
\iu(u)&\geq&\oh\|u\|^2-\frac{\lambda}{2}\int\ap u^2-\ott\int\ap f(\x u)u+\mu\int\am G(\x u)\nonumber\\
	&=&\left(\oh-\ott\right)\left(\|\ut+\uh+\ub\|^2-{\lambda}\int\ap (\ut+\uh+\ub)^2\right)\nonumber\\
	&&-\lambda\left(1-\ott\right)\int\ap (\ut+\uh+\ub)\bu-\ott\int\ap f(\x u)\bu\nonumber\\
	&&+\left(\oh\|\bu\|^2-\frac{\lambda}{2}\int\ap \bu^2+\mu\int\am G(\x \bu)\right)\nonumber\\ 
	&\geq&\left(1-\frac{2}{\theta}\right)(1-\Lambda)\|\ut+\uh+\ub\|^2+o(1).\nonumber
\end{eqnarray}
For the last inequality we have used
condition (a) and Lemma~\ref{L_one}. 
This lower bound for $\iu(u)$ implies 
$$
\|\ut+\uhp\|\leq\overline{R}+o(1).
$$
So
$$
\|\ut+\uhp\|\leq\hat{R}
$$
for sufficiently large $\mu$.
\end{proof}
The next lemma will be used in connection with (${\cal{N}}_{v}$):
\begin{Lem}\label{dois}
Let $0<\delta<1/2$. 
If $\mu$ is sufficiently large, then
for any $u\in\nmu$ with $\iu(u)<\inf_{\nmu}\iu+\delta$,
$$
\|\uhm\|\leq 2\sqrt{\delta/\gamma}\quad\mbox{and}\quad \|\ub\|\leq 2\sqrt{\delta/\gamma}.
$$
\end{Lem}
\begin{proof}
Notice that 
$$
u\in\nmu\quad\Rightarrow\qquad 
u+\uhm\ {\rm and}\ u-\ub\ {\rm satisfy}\ {\rm (}{\cal{N}}_{i}{\rm )}.
$$
From (a) and Lemma~\ref{L_one},
\begin{eqnarray*}
\int\ap f(\x u)\uhm-\int\ap f(\x -\uhm)\uhm&=&\int\left(\ap\int_0^1f^\prime(\x s\bu-\uhm)\,ds\,\bu\uhm\right)\\
&=&o(1).
\end{eqnarray*}
Let $0<\delta<1/2$.
A simple consequence of (\ref{p1}) is that for $u\in\nmu$ with 
$\iu(u)<\inf_{\nmu}\iu+\delta$, with $\mu$ sufficiently large so that
$u+\uhm\in\nmu$,
\begin{eqnarray}
\inf_{\nmu}\iu&\leq&\iu(u+\uhm)\nonumber\\
	&=&\iu(u)-\left(\oh\|\uhm\|^2-\frac{\lambda}{2}\int\ap(\uhm)^2-\int\ap F(\x -\uhm)\right)+o(1)\nonumber\\
	&\leq&\iu(u)-\gamma\|\uhm\|^2+o(1)\nonumber\\
	&<&\inf_{\nmu}\iu+\delta-\gamma\|\uhm\|^2+o(1).\label{uh}
\end{eqnarray}
Similarly,
if $\mu$ is sufficiently large, then $u-\ub\in\nmu$ and 
\begin{equation}
\inf_{\nmu}\iu\leq\iu(u-\ub)\leq\iu(u)-\gamma\|\ub\|^2+o(1)<\inf_{\nmu}\iu+\delta-\gamma\|\ub\|^2+o(1).
\label{ub}
\end{equation}
Inequalities (\ref{uh}) and (\ref{ub}) imply Lemma~\ref{dois}.
\end{proof}
The next lemma will be used in connection with (${\cal{N}}_{vi}$):
\begin{Lem}\label{kappa2}
There exists a positive constant $\kappa_2$ such that for all $\mu$ sufficiently large,
$$\textstyle
u\in\nmu\ \ \Rightarrow\ \
\min\left\{\nutp,
\nutm,\nuhp\right\}\geq\kappa_2^{1/2}.
$$
\end{Lem}
\begin{proof}
Consider again $u\in\nmu$ and $w$ equal to one of the three functions
$\utp$, $-\utm$ or $\uhp$. Let $\varsigma$ be such that $\op=\frac{\varsigma}{2}+\frac{1-\varsigma}{2^*}$.
From Lemma~\ref{L_one} and (\ref{ap3}), 
\begin{eqnarray*}
\|w\|^2&=&\lambda\int\ap uw+\int\ap f(\x u)w=\lambda\int\ap w^2+\int\ap f(\x w)w+o(1)\nonumber\\
	&\leq&{\textstyle\frac{3\Lt+1}{4}}\|w\|^2+C_1\left(\int\ap w^2\right)^{p\varsigma/2}
	\left(\sup\ap|w|_{2^*}^{2^*}\right)^{p(1-\varsigma)/2^*}+o(1).
\end{eqnarray*}
Hence, Lemma~\ref{kappa2} follows from Lemma~\ref{L_two}.
\end{proof}

For $u$ and $w$ as above,
consider the function $\g:\R^+\to\R$, defined by
$$ 
\g(t;w)=\int\ap\frac{f(\x tw)}{t}w=\int\ap\frac{f(\x tw)}{tw}w^2,
$$ 
with the understanding that $f(\x u)/u=0$ for $u=0$.
Henceforth the letter $C$ denotes a constant which may differ from line to line.
We examine some simple properties of $\g$:
\begin{Claim}
If $\mu$ is sufficiently large, then (i)
the function $\g(t;w)\to+\infty$ as $t\to+\infty$, uniformly in $\mu$ and $u\in\nmu$, 
(ii) $\g(t;w)\to\g_0(w)$ as $t\to 0$, (iii) the function
$\g$ is strictly increasing, (iv) there exists $\kappa_0$, independent of
$\mu$ and $u\in\nmu$, such that $\g(1;w)-\g_0(w)\geq\kappa_0$,
and (v) there exists $\kappa'_0$, independent of
$\mu$ and $u\in\nmu$, such that
$\g'(t;w)\geq\kappa'_0$ for $t\in[\eta,1/\eta]$; here $0<\eta<1$ is fixed.
\end{Claim}
\begin{proof}
As in (\ref{s}),
consider the set 
\begin{equation}\label{SS3}
\str:=\{x\in\Omega^+:\:\mbox{dist}(x,\R^N\setminus\Omega^+)<1/\underline{n}\},
\end{equation}
where $\underline{n}$ is large enough so that 
$$
|\str|\leq\textstyle\left(\frac{\kappa_2c_{2^*}^2}{2R^2\sup\ap}\right)^{2^*/(2^*-2)}.
$$
We have
\begin{equation}\label{S3}
\int_{\Omega^+\setminus \str}\ap w^2\geq\kappa_2-\int_{\str}\ap w^2
\geq\kappa_2-\sup\ap\frac{R^2}{c_{2^*}^2}|\str|^{(2^*-2)/2^*}\geq\frac{\kappa_2}{2}.
\end{equation}
Let $0<\delta<1$ be fixed. From (b), there exists a constant $c_\delta$ such that
$$ 
\frac{f(\xx u)}{u}\geq c_\delta|u|^{\theta-2}-\delta
\quad\mbox{for}\ u\in\R\ \mbox{and}\ x\in\Omega^+\setminus \str.
$$ 
This gives a lower bound for $\g(t;w)$:
\begin{eqnarray*}
\g(t;w)&\geq&c_\delta t^{\theta-2}\int_{\Omega^+\setminus \str}\ap w^\theta
	-\delta\int_{\Omega^+\setminus \str}\ap w^2\nonumber\\
	&\geq&c_\delta t^{\theta-2}|\ap|_1^{1-\theta/2}
	\left(\int_{\Omega^+\setminus \str}\ap w^2\right)^{\theta/2}
	-\delta\int\ap w^2 \nonumber\\
	&\geq&c_\delta t^{\theta-2}|\ap|_1^{1-\theta/2}
	\frac{\kappa_2^\theta}{2^\theta}-C.
\end{eqnarray*}
The function $\g(t;w)\to+\infty$ as $t\to+\infty$, uniformly in $u\in\nmu$ and $\mu$, with $\mu$ large.
On the other hand, by (a) and the Dominated Convergence Theorem,
$$
\g(t;w)\to \g_0(w):=\int\ap f^\prime(\x 0)w^2\quad\mbox{as}\ t\to 0.
$$
The function $\g$ is strictly increasing
as (c) implies $u\frac{d}{du}\bigl(\frac{f(\x u)}{u}\bigr)>0$.
There exists $\kappa_0>0$, independent of $\mu$ (large) and $u\in\nmu$, such that
\begin{equation}
\g(1;w)-\g_0(w)=\int\ap f(\x w)w-\int\ap f^\prime(\x 0)w^2\geq\kappa_0.\label{g}
\end{equation}
This is a consequence of
\begin{Claim}\label{claim2}
Let $0<\eta<1$ be fixed.
There exists $\kappa_0^\prime>0$, independent of $\mu$ (large) and $u\in\nmu$, such that
\begin{equation}\label{prime}
\g'(t;w)=
\frac{1}{t}\int\ap\left(f^\prime(\x tw)-\frac{f(\x tw)}{tw}\right)w^2
\geq \kappa_0^\prime,\quad\mbox{for}\ t\in[\eta,1/\eta].
\end{equation}
\end{Claim}
\noindent The proof of Claim~\ref{claim2} is given in the Appendix.
\end{proof}
Finally,
the next lemma will be used in connection with (${\cal{N}}_{vii}$):
\begin{Lem}\label{smallbu}
Let $\delta>0$. If 
$\delta$ is sufficiently small and 
$\mu$ is sufficiently large, then for any $u\in\nmu$
with $\iu(u)<\inf_{\nmu}\iu+\delta$,
$$
\|\bu\|\leq 2\sqrt{\delta/\gamma}\quad\mbox{and}\quad\mu\int\am G(\x \bu)\leq 2\delta.
$$
\end{Lem}
\begin{proof}
For $u\in\nmu$, we do not expect $u-\bu$ to belong
to $\nmu$ because this function might not satisfy (${\cal{N}}_{i}$).
We wish to determine $\rt$, $\st$ and $\that$ such that
$$\check{u}=\rt\utp-\st\utm+\that\uhp-\uhm+\ub$$
satisfies (${\cal{N}}_{i}$). Since $u\in\nmu$,
$$
\begin{array}{lcl}
\displaystyle\|\utp\|^2-\lambda\int\ap(\utp)^2&=&\displaystyle\int\ap f(\x \utp)\utp+o(1),\\
\displaystyle\|\utm\|^2-\lambda\int\ap(\utm)^2&=&\displaystyle-\int\ap f(\x -\utm)\utm+o(1),\\
\displaystyle\|\uhp\|^2-\lambda\int\ap(\uhp)^2&=&\displaystyle\int\ap f(\x \uhp)\uhp+o(1).
\end{array}
$$
The function $\check{u}$ satisfies (${\cal{N}}_{i}$) if
$$
\begin{array}{lcl}
\displaystyle\rt\left(\|\utp\|^2-\lambda\int\ap(\utp)^2\right)&=&\displaystyle\int\ap f(\x \rt\utp)\utp,\\
\displaystyle\st\left(\|\utm\|^2-\lambda\int\ap(\utm)^2\right)&=&\displaystyle-\int\ap f(\x -\st\utm)\utm,\\
\displaystyle\that\left(\|\uhp\|^2-\lambda\int\ap(\uhp)^2\right)&=&\displaystyle\int\ap f(\x \that\uhp)\uhp,
\end{array}
$$
or
$$
\begin{array}{lclcl}
\g(\rt;\utp)&=&\g(1;\utp)+o(1)&>&\g_0(\utp),\\
\g(\st;-\utm)&=&\g(1;-\utm)+o(1)&>&\g_0(-\utm),\\
\g(\that;\utp)&=&\g(1;\uhp)+o(1)&>&\g_0(\uhp).
\end{array}
$$
The last three inequalities (which hold for large $\mu$) follow from
(\ref{g}). The properties of the
functions $\g$ guarantee that the desired
$\rt$, $\st$ and $\that$ do exist and are unique. 
The lower bound (\ref{prime}) allows us to conclude
\begin{equation}\label{rst}
\rt=1+o(1),\quad \st=1+o(1),\quad\that=1+o(1).
\end{equation}
Let $0<\delta<\min\left\{1/2,(1-2/\theta)(1-\Lambda)R^2-\iu(v)\right\}$.
Suppose $u\in\nmu$ with $\iu(u)<\inf_{\nmu}\iu+\delta$.
Choose $0<\overline{R}<\hat{R}<R$ such that
$
\inf_{\nmu}\iu+\delta\leq (1-2/\theta)(1-\Lambda)\overline{R}^2
$.
If $\mu$ is sufficiently large, $\check{u}\in\nmu$ 
because of (\ref{rst}), Lemmas~\ref{overr} and \ref{dois}, and
$\iu(\check{u})\leq\iu(u)+o(1)$.
We obtain
\begin{eqnarray}
\inf_{\nmu}\iu&\!\!\leq\!\!&\iu(\check{u})\leq\iu(\ut+\uh+\ub)+o(1)\nonumber\\
	&\!\!=&\!\!\iu(u)-\oh\|\bu\|^2+\frac{\lambda}{2}\int\ap\bu^2
	+\int\ap F(\x \bu)-\mu\int\am G(\x \bu)+o(1)\nonumber\\
	&\!\!\leq\!\!&\inf_{\nmu}\iu+\delta-\gamma\|\bu\|^2-\mu\int\am G(\x \bu)+o(1).\label{lbu}
\end{eqnarray}
We have used (\ref{rst}).
Inequality (\ref{lbu}) implies Lemma~\ref{smallbu}.
\end{proof}
\section{Existence of least energy solutions}\label{section3}

\setcounter{equation}{0}
For each $u\in\nmu$ we define a
3-dimensional manifold $\mm$ with global
chart 
$\varphi:\R^3_+\to\hozo$, given by
\begin{equation}\label{twelve}
\varphi(\rt,\st,\that)=\rt\utp-\st\utm+\that\uhp-\uhm+\ub+\bu.
\end{equation}
Note $\varphi(1,1,1)=u$. 
\begin{Lem}\label{max} 
If $\mu$ is sufficiently large,
the functional $\left.\iu\right|_{\mm}$ 
has a unique absolute maximum.
This maximum is strict and attained at $u$.
\end{Lem}
\begin{proof}
To evaluate the functional $\left.\iu\right|_{\mm}$, we introduce
$h:\R^3_+\to\R$,
\begin{eqnarray}
h(\rt,\st,\that)&:=&\iu\circ\varphi\,(\rt,\st,\that)\label{defh}\\ 
&=&\frac{\rt^2}{2}\left(\left\|\utp\right\|^2-\lambda\int\ap(\utp)^2\right)+
\frac{\st^2}{2}\left(\left\|\utm\right\|^2-\lambda\int\ap(\utm)^2\right)\nonumber\\
&&+\frac{\that^2}{2}\left(\left\|\uhp\right\|^2-\lambda\int\ap(\uhp)^2\right)
-\rt\lambda\int\ap\utp\bu+\st\lambda\int\ap\utm\bu\nonumber\\
&&-\that\lambda\int\ap\uhp\bu
-\int\ap F(\x \rt\utp+\bu)-\int\ap F(\x \bu-\st\utm)\nonumber\\
&&-\int\ap F(\x \that\uhp+\bu)+C_2,\nonumber
\end{eqnarray}
with $C_2$ a constant.
From (${\cal{N}}_{i}$), $\nabla h(1,1,1)=0$.
Let $\nu$ designate one of $\rt$, $\st$ or $\that$, and 
accordingly let $w$ designate $\utp$, $-\utm$ or $\uhp$.
When $\nu=1$ and no matter what the values of the other two variables,
\begin{eqnarray}
\left.\frac{\partial^2 h}{\partial\nu^2}\right|_{\nu=1}&=&\|w\|^2-\lambda\int\ap w^2-\int\ap f^\prime(\x w+\bu)w^2\nonumber\\
&=&\|w\|^2-\lambda\int\ap w^2-\int\ap f^\prime(\x w)w^2+o(1).\label{neglet}
\end{eqnarray}
Indeed, (\ref{neglet}) follows from
\begin{Claim}\label{claim3}
For any positive $\delta$, there exists $\mu_\delta$ such that, for all $\mu>\mu_\delta$
and $u\in\nmu$,
$$
\left|\int\ap f^\prime(\x w+\bu)w^2-\int\ap f^\prime(\x w)w^2\right|\leq\delta.
$$
\end{Claim}
\noindent We leave the simple proof to the reader.
Returning to the computation of the second derivative in (\ref{neglet}),
we now use (${\cal{N}}_{i}$) and Lemma~\ref{L_one}, and afterwards (\ref{prime}) for $t=1$:
\begin{eqnarray*}
\left.\frac{\partial^2 h}{\partial\nu^2}\right|_{\nu=1}&=&
-\int\ap f^\prime(\x w)w^2+\int\ap f(\x w)w+o(1)\\
&\leq& -\kappa_0^\prime+o(1)\\ &\leq&-\kappa_0^\prime/2,
\end{eqnarray*}
for $\mu$ sufficiently large. Furthermore, we can find 
$\underline{\nu}<1<\overline{\nu}$, 
independent of $u\in\nmu$ for $\mu$ large, such that
\begin{equation}\label{unov}
\nu\in[\underline{\nu},\overline{\nu}]\quad\Rightarrow\quad\frac{\partial^2 h}{\partial\nu^2}\leq-
\frac{\kappa_0^\prime}{4}.
\end{equation}
So the function $h$ has a strict local maximum at
$(1,1,1)$. The function $h$ differs by an $o(1)$ from $\underline{h}:
\R^3_+\to\R$,
\begin{eqnarray*}
\underline{h}(\rt,\st,\that)&:=&\frac{\rt^2}{2}\left(\left\|\utp\right\|^2-\lambda\int\ap(\utp)^2\right)
+\frac{\st^2}{2}\left(\left\|\utm\right\|^2-\lambda\int\ap(\utm)^2\right)\\
&&+\frac{\that^2}{2}\left(\left\|\uhp\right\|^2-\lambda\int\ap(\uhp)^2\right)+C_2\\
&&-\int\ap F(\x \rt\utp)-\int\ap F(\x -\st\utm)-\int\ap F(\x \that\uhp).
\end{eqnarray*}
For large $\mu$, $\underline{h}$ also must have a strict local maximum
in $[\underline{\nu},\overline{\nu}]^3$, say at $(\rt_1,\st_1,\that_1)$ (dependent
on $\mu$ and $u$, of course).
It is simple to check using (c) that if
$\partial\underline{h}/\partial\nu=0$, then
$\partial^2\underline{h}/\partial\nu^2<0$. This implies that
$\partial\underline{h}/\partial\nu>0$ for $\nu<\check{\nu}$ and 
$\partial\underline{h}/\partial\nu<0$ for $\nu>\check{\nu}$.
Again, we use the fact that $h$ is uniformly close to $\underline{h}$
to see that the maximum of $h$ at $(1,1,1)$ is unique and absolute.
We have proved Lemma~\ref{max}.
\end{proof}
\begin{Prop}\label{propo}
Let $(u_n)$ be a minimizing sequence for $\iu$ restricted to $\nmu$.
Then, modulo a subsequence, for sufficiently large $\mu$, $u_n\to u$ in $\hozo$ and $u$ is a minimizer.
\end{Prop}
\begin{proof}
Let $(u_n)$ be a minimizing sequence for $\iu$ restricted to $\nmu$,
$u_n\rightharpoonup u$ in $\hozo$. Let $w_n$ be $\utpn$, $-\utmn$ or
$\uhpn$ and, accordingly, let $w$ be $\utp$, $-\utm$ or $\uhp$
and $\nu$ be $\rt$, $\st$ or $\that$.
Suppose that 
\begin{equation}\label{gap}
\|w\|<\liminf\|w_n\|.
\end{equation}
Lemma~\ref{kappa2} gives
\begin{equation}\label{lbw}
\|w\|\geq\kappa_2^{1/2}\Lambda_1^{1/2}.
\end{equation}
The function $u$ will not
satisfy (${\cal{N}}_{i}$) because
$$
\|w\|^2-\lambda\int\ap w^2-\lambda\int\ap\bu w-\int\ap f(\x w+\bu)w<0.
$$ 
We define the value 
$$
\nu_0=\textstyle\frac{1}{R}\left(\frac{\gamma c_p^{p}}
{\sup\ap C_1}\right)^{1/(p-2)}=\frac{\rho}{R}.
$$
As in (\ref{ap3}),
$$
\begin{array}{l}
\displaystyle\nu_0^2\left(\|w\|^2-\lambda\int\ap w^2\right)-\nu_0\left(\lambda\int\ap\bu w-\int\ap f(\x \nu_0 w+\bu)w\right)\\
\displaystyle\qquad\qquad =\nu_0^2\left(\|w\|^2-\lambda\int\ap w^2-\int\ap\frac{f(\x \nu_0 w)}{\nu_0}w\right)\\
\displaystyle\qquad\qquad\quad-\nu_0\left(\lambda\int\ap\bu w+\int\ap f(\x \nu_0 w+\bu)w-\int\ap f(\x \nu_0 w)w\right)\\
\vspace{1mm}
\displaystyle\qquad\qquad\geq\nu_0^2\left(\|w\|^2-{\textstyle\frac{3\Lt+1}{4}}\|w\|^2-C_1\nu_0^{p-2}\int\ap|w|^p\right)+o(1)\\
\displaystyle\qquad\qquad\geq 2\gamma\nu_0^2\kappa_1^2+o(1)\\
\displaystyle\qquad\qquad>0,
\end{array}
$$
for large $\mu$. By continuity, there will exist $\nu_1\in\,]\nu_0,1[$\, such that 
$$\nu_1\left(\|w\|^2-\lambda\int\ap w^2\right)-\lambda\int\ap\bu w-\int\ap f(\x \nu_1 w+\bu)w=0.$$
(The value of $\nu_1$ depends on $w$, but $\nu_0$ is fixed.)
Hence there exist
\begin{equation}\label{range}
\rt_1,\ \st_1,\ \that_1\in\ ]\nu_0,1]
\end{equation}
such that
$$\check{u}:=\rt_1\utp-\st_1\utm+\that_1\uhp-\uhm+\ub+\bu$$
satisfies (${\cal{N}}_{i}$). It also satisfies (${\cal{N}}_{ii}$) 
because of (\ref{lbw}).
We estimate the energy of $\check{u}$ using
(\ref{gap}) and Lemma~\ref{max} applied to $u_n$:
\begin{eqnarray}
\iu(\check{u})&<&\liminf\iu(\rt_1\utpn-\st_1\utmn+\that_1\uhpn-\uhmn+\ubn+\bun)\label{contra}\\
&\leq&\lim\iu(u_n)\nonumber\\
&=&\inf\left.\iu\right|_{\nmu}.\nonumber
\end{eqnarray}
So $\check{u}$ satisfies (${\cal{N}}_{iii}$).
The function $\check{u}$ satisfies (${\cal{N}}_{iv}$)
because of (\ref{range}), and it clearly satisfies (${\cal{N}}_{v}$). 
It satisfies (${\cal{N}}_{vi}$) for large $\mu$ because
of Lemmas~\ref{L_one} and~\ref{kappa2} and of the strong convergence
of $u_n$ to $u$ in $L^2(\Omega)$. 
Applying Lemma~\ref{smallbu} to $u_n$ for large $n$, with 
$\delta=\nu_0^2\kappa_2{\Lambda}_1/(4\gamma)$ 
and $\mu$ sufficiently large,  and using (\ref{lbw}) and the weak
lower semi-continuity of $\|\bun\|$, we obtain
$\|\bu\|\leq\nu_0\kappa_2^{1/2}{\Lambda}_1^{1/2}\leq\nu_0\min\{\|\utp\|,\|\utm\|,\|\uhp\|\}$.
So $\check{u}$ also satisfies (${\cal{N}}_{vii}$). In conclusion,
$\check{u}$ belongs to $\nmu$ and inequality (\ref{contra}) is impossible.
Thus, $\|w\|=\liminf\|w_n\|$ and $u\in\nmu$ is a minimizer of $\iu$ restricted
to $\nmu$. In fact, if $\|u\|$ were to be smaller than $\liminf\|u_n\|$,
due to a drop in $\|\uhmn\|$, $\|\ubn\|$ or $\|\bun\|$ upon passing to the
limit, then we would still have strict inequality in (\ref{contra}),
and again a contradiction. We have proved Proposition~\ref{propo}.
\end{proof}
\begin{Prop}\label{grau}
If $\mu$ is sufficiently large, every minimizer of $\iu$ on
$\nmu$ is a critical point of $\iu$.
\end{Prop}
\begin{proof}
Let $\mu$ be large enough so that Proposition~\ref{propo} holds.
Let $u$ be a minimizer of $\iu$ restricted to $\nmu$.
Consider the maps $\jtp$, $\jtm$, $\jhp:\hozo\to\R$ defined by
$$ 
\jtp(z)=\iu'(z)\ztp,\quad\jtm(z)=-\iu'(z)\ztm,\quad\jhp(z)=\iu'(z)\zhp,
$$ 
and
$J_\mu:[\underline{\nu},\overline{\nu}]^3\to\R^3$ defined by
$$
J_\mu(\rt,\st,\that)=(\jtp,\jtm,\jhp)\circ\varphi\,(\rt,\st,\that)=
\left(\rt\frac{\partial h}{\partial\rt},\st\frac{\partial h}{\partial\st},
\that\frac{\partial h}{\partial\that}\right).
$$
Here the maps $\varphi$ and $h$ are the ones corresponding to $u$ as in (\ref{twelve})
and (\ref{defh}).
Using (\ref{unov}) and $\nabla h(1,1,1)=0$, we can find $\nu_2$ and $\nu_3$, 
independent of $\mu$, with
$\underline{\nu}\leq\nu_2<1$ and $1<\nu_3\leq\overline{\nu}$, such that
\begin{equation}\label{desce}
(\rt,\st,\that)\in[\nu_2,\nu_3]^3\quad\Rightarrow\quad\frac{\partial (\jtp\circ\varphi)}{\partial\rt},\
\frac{\partial (\jtm\circ\varphi)}{\partial\st},\ \frac{\partial (\jhp\circ\varphi)}{\partial\that}
\leq-\frac{\kappa_0^\prime}{8}.
\end{equation}
It follows that on the boundary of
$[\nu_2,\nu_3]^3$ either one of the components of $J_\mu$ is greater than 
$\frac{\kappa_0^\prime}{8}(1-\nu_2)$,
or one of the components of $J_\mu$ is less than $-\frac{\kappa_0^\prime}{8}(\nu_3-1)$ and
$$
\mbox{deg}\,\left(J_\mu,[\nu_2,\nu_3]^3,0\right)=-1.
$$
Suppose that $\iu^\prime(u)\neq 0$. Let 
$B_{\hat{\rho}}(u):=\{z\in\hozo:\:\left\|z-u\right\|<{\hat{\rho}}\}$.
Choose ${\hat{\rho}}>0$ satisfying $\iu^\prime(z)\neq 0$ for all $z\in B_{\hat{\rho}}(u)$,
\begin{equation}\label{peq}
\hat{\rho}<\mbox{dist}\left(u,\varphi\left(\R^3_+\setminus[\nu_2,\nu_3]^3\right)\right),
\end{equation}
and so that conditions (${\cal{N}}_{ii}$)$-$(${\cal{N}}_{vii}$) hold for all $z\in B_{\hat{\rho}}(u)$.
This is possible because $u\in\nmu$ ((${\cal{N}}_{ii}$)),
$\iu(u)\leq\iu(v)$
((${\cal{N}}_{iii}$)),
of Lemma~\ref{overr} ((${\cal{N}}_{iv}$)),
of Lemma~\ref{dois} ((${\cal{N}}_{v}$)),
of Lemmas~\ref{L_one} and \ref{kappa2} ((${\cal{N}}_{vi}$)),
and of Lemmas~\ref{L_two} and \ref{smallbu} ((${\cal{N}}_{vii}$)).
Note that the choice of ${\hat{\rho}}$ might depend on $\mu$.
Let $\phi:\hozo\to[0,1]$ be Lipschitz,
$\phi=1$ on $B_{{\hat{\rho}}/2}(u)$ and $\phi=0$ on
$\hozo\setminus B_{\hat{\rho}}(u)$ and
let $K_\mu:B_{\hat{\rho}}(u)\to\hozo$ be a pseudogradient vector field for
$\iu^\prime$ on $B_{\hat{\rho}}(u)$.
Consider the Cauchy problem
$$ 
\left\{\begin{array}{l}
\displaystyle\frac{d\eta}{d\tau}=-\phi(\eta)K_\mu(\eta),\\
\vspace{-4mm}\\
\eta(0)=z,\end{array}\right.
$$ 
for $z\in\hozo$; by definition, $\phi K_\mu$ is zero outside
$B_{\hat{\rho}}(u)$. We denote the solution of this Cauchy problem
by $\eta(\tau;z)$. 
For $\tau>0$, let
$$
\varphi_\tau(\rt,\st,\that)=\eta(\tau;\varphi(\rt,\st,\that)).
$$
Each $\varphi_\tau$ is continuous and, due to
(\ref{peq}), 
$$
\left.\varphi_\tau\right|_{\partial\left([\nu_2,\nu_3]^3\right)}=
\left.\varphi\right|_{\partial\left([\nu_2,\nu_3]^3\right)}
$$
and so
$$
\mbox{deg}\,\left(J_\mu^\tau,\left[\nu_2,\nu_3\right]^3,0\right)
=\mbox{deg}\,\left(J_\mu,[\nu_2,\nu_3]^3,0\right)=
-1,
$$
where
$$
J_\mu^\tau(\rt,\st,\that):=
\left(\jtp,\jtm,\jhp\right)\circ\varphi_\tau(\rt,\st,\that).
$$
It follows that there exists some $(\rt_1,\st_1,\that_1)\in\,]\nu_2,\nu_3[^3$,
with $\varphi_\tau(\rt_1,\st_1,\that_1)$ satisfying (${\cal{N}}_{i}$).
The function $\varphi_\tau(\rt_1,\st_1,\that_1)$ has to belong to $B_{\hat{\rho}}(u)$
as outside $B_{\hat{\rho}}(u)$ the maps $\varphi$ and $\varphi_\tau$ coincide,
and $\varphi$ only satisfies (${\cal{N}}_{i}$) in $[\nu_2,\nu_3]^3$
at the point $(1,1,1)$. This is a consequence of (\ref{desce}).
But on $B_{\hat{\rho}}(u)$
conditions (${\cal{N}}_{ii}$)$-$(${\cal{N}}_{vii}$) hold,
so  $\varphi_\tau(\rt_1,\st_1,\that_1)$ belongs to $\nmu$.
By Lemma~\ref{max}, the maximum of $\iu\circ\varphi$ is
strict and attained at $(1,1,1)$.
For $\tau>0$, $\max\iu\circ\varphi_\tau<
\iu(u)=\left.\min\iu\right|_{\nmu}$. This contradicts 
$\varphi_\tau(\rt_1,\st_1,\that_1)\in\nmu$. We have proved Proposition~\ref{grau}.
\end{proof}
\begin{altproof}{Theorem~\ref{thm}}
By Propositions~\ref{propo} and \ref{grau}, there exists $\check{\mu}$ such that
for $\mu>\check{\mu}$ the equation (\ref{one})
has an $\hozo$ weak solution $u_\mu$.
Suppose $\mu_n\to+\infty$ and $u_{\mu_n}$ is a minimizer of $I_{\mu_n}$
restricted to ${\cal{N}}_{\mu_n}$. Modulo a subsequence,
$$
u_{\mu_n}\rightharpoonup u\quad\mbox{in}\ \hozo.
$$
It is clear from Lemmas~\ref{dois} and \ref{smallbu} that
$$
u=\ut+\uhp,
$$
and
\begin{eqnarray}
I_{\mu_n}(u_{\mu_n})&=&\oh\|\tilde{u}_{\mu_n}\|^2-\frac{\lambda}{2}\int\ap(\tilde{u}_{\mu_n})^2+
\oh\|\hat{u}^+_{\mu_n}\|^2-\frac{\lambda}{2}\int\ap(\hat{u}^+_{\mu_n})^2\nonumber\\
&&-\int\ap F(\x \tilde{u}_{\mu_n})-\int\ap F(\x \hat{u}^+_{\mu_n})+o(1).\label{iinf}
\end{eqnarray}
Obviously from Lemma~\ref{kappa2}
$$
\utp, \utm, \uhp\not\equiv 0.
$$
Suppose that either one of the two inequalities
\begin{equation}\label{cai}
\|\ut\|<\liminf\|\tilde{u}_{\mu_n}\|\quad\mbox{or}\quad
\|\uhp\|<\liminf\|\hat{u}^+_{\mu_n}\|
\end{equation}
is satisfied. Then 
\begin{eqnarray*}
I_0(u)&=&\oh\|\tilde{u}\|^2-\frac{\lambda}{2}\int\ap\tilde{u}^2+
\oh\|\hat{u}^+\|^2-\frac{\lambda}{2}\int\ap(\hat{u}^+)^2\\ 
&&-\int\ap F(\x \tilde{u})-\int\ap F(\x \hat{u}^+)\\ 
&<&\liminf I_{\mu_n}(u_{\mu_n}).
\end{eqnarray*}
We can argue as above to prove that there exists
$(\rt,\st,\that)\in\,]0,1]^3\setminus\{(1,1,1)\}$
such that $\check{u}:=\rt\utp-\st\utm+\that\uhp$
satisfies (${\cal{N}}_{i}$).
The function $\check{u}$ also satisfies (${\cal{N}}_{ii}$).
Using first the hypothesis that one of the inequalities (\ref{cai}) is strict,
then Lemmas~\ref{dois} and \ref{smallbu}, and finally
Lemma~\ref{max} applied to $u_{\mu_n}$,
\begin{eqnarray}
I_{\mu_n}(\check{u})&<&\liminf I_{\mu_n}\left(\rt\tilde{u}^+_{\mu_n}-
\st\tilde{u}^-_{\mu_n}+\that\hat{u}^+_{\mu_n}\right)\nonumber\\
&=&\liminf I_{\mu_n}\left(\rt\tilde{u}^+_{\mu_n}-
\st\tilde{u}^-_{\mu_n}+\that\hat{u}^+_{\mu_n}-\hat{u}^-_{\mu_n}+
\bar{u}_{\mu_n}+\underline{u}_{\,\mu_n}\right)\nonumber\\
&\leq&\liminf I_{\mu_n}\left(u_{\mu_n}\right)\nonumber\\
&=&\liminf\min\left.I_{\mu_n}\right|_{{\cal{N}}_{\mu_n}}.\label{infinito}
\end{eqnarray}
The function $\check{u}$ also satisfies (${\cal{N}}_{iii}$).
Obviously, $\check{u}$ satisfies (${\cal{N}}_{iv}$)$-$(${\cal{N}}_{vii}$).
Thus $\check{u}$ belongs to ${\cal{N}}_{\mu_n}$. This contradicts
(\ref{infinito}) and proves that
$$
u_{\mu_n}\to u\quad\mbox{in}\ \hozo.
$$
This proves (\ref{fifteen}).
Also, from (\ref{iinf}),
\begin{equation}\label{key}
I_0(u)=\lim I_{\mu_n}\left(u_{\mu_n}\right).
\end{equation}
The proof of Theorem~\ref{thm} will be complete once we prove
\begin{Claim}\label{last}
Let $u$ be as in {\rm (\ref{fifteen})}.
The function $u|_{\ot}$ is a least energy nodal solution in $H^1_0(\ot)$ of\/ {\rm (\ref{one})},
and the function $u|_{\ohat}$ is a least energy positive solution in $H^1_0(\ohat)$ of\/ {\rm (\ref{one})}.
\end{Claim}
\begin{proof} 
Suppose $\upsilon\in\hozo$ is
such that $\upsilon|_{\Omega^+}$
a solution of (\ref{one}), 
$\upsilon|_{\ob}$ and $\upsilon|_{\Omega^-}$ are zero,
$\upsilon|_{\ot}$ is nodal,
$\upsilon|_{\ohat}$ is positive,
and either 
$$
I_0\left(\upsilon|_{\ot}\right)<I_0\left(u|_{\ot}\right)\quad\mbox{or}
\quad I_0\left(\upsilon|_{\ohat}\right)<I_0\left(u|_{\ohat}\right).
$$
Because $\Omega$, $\ot$ and $\ohat$ are Lipschitz,
$\tilde{\upsilon}$ coincides with $\upsilon|_{\ot}$ in $\ot$,
and $\hat{\upsilon}$ coincides with $\upsilon|_{\ohat}$ in $\ohat$.
Without loss of generality, we may also assume 
$$
I_0\left(\upsilon|_{\ot}\right)\leq I_0\left(u|_{\ot}\right)\quad\mbox{and}
\quad I_0\left(\upsilon|_{\ohat}\right)\leq I_0\left(u|_{\ohat}\right).
$$
Multiplying both sides of (\ref{one}) by $\tilde{\upsilon}^+$ and integrating,
by $\tilde{\upsilon}^-$ and integrating,
and by $\hat{\upsilon}^+$ and integrating, we find
$I_0'(\upsilon)\left(\tilde{\upsilon}^+\right)=
I_0'(\upsilon)\left(\tilde{\upsilon}^-\right)=
I_0'(\upsilon)\left(\hat{\upsilon}^+\right)=0
$.
Note that
$$
\left(1-\frac{2}{\theta}\right)(1-\Lambda)\|\tilde{\upsilon}+\hat{\upsilon}^+\|^2\leq I_0(\upsilon)<
I_0(u)\leq\left(1-\frac{2}{\theta}\right)(1-\Lambda)R^2.
$$
The function $\upsilon\in\nmu$ for all $\mu$. From (\ref{key}) we arrive
at the contradiction
$$
I_0(\upsilon)<I_0(u)=\lim I_{\mu_n}\left(u_{\mu_n}\right)=
\lim\min\left.I_{\mu_n}\right|_{{\cal{N}}_{\mu_n}}.
$$
Therefore,
$$
I_0\left(\upsilon|_{\ot}\right)\geq I_0\left(u|_{\ot}\right)\quad\mbox{and}
\quad I_0\left(\upsilon|_{\ohat}\right)\geq I_0\left(u|_{\ohat}\right).
$$
We have proved Claim~\ref{last}.
\end{proof}
\noindent The proof of Theorem~\ref{thm} is complete.
\end{altproof}

\section{Appendix}\label{section5}

\setcounter{equation}{0}
In this Appendix we give a direct proof of Lemma~\ref{L_one}
and of Claim~\ref{claim2}. 

\begin{altproof}{{\bf Lemma~\ref{L_one}}}
Let $\delta>0$, $\zeta$, $\varsigma$ be such that
$$\textstyle
\frac{1}{p}=\frac{\zeta}{\vartheta}+\frac{1-\zeta}{2^*},\quad\quad
\frac{1}{p}=\frac{\varsigma}{2}+\frac{1-\varsigma}{2^*},
$$
$$
\hat{C}=\textstyle\left(\frac{c_{2^*}^{p(1-\varsigma)}}
{2C_T^{p\varsigma}R^{p(1-\varsigma)}}\right)^{2/\varsigma},
$$
$$\overline{\delta}=\textstyle\frac{\hat{C}^{1/p}{\delta^{2/(p\varsigma)}}}{2R^2}
\quad\mbox{and}\quad
\hat{\delta}=\textstyle
\left(
\frac{c_{2^*}^{p(1-\zeta)}}{R^{p(1-\zeta)}}
\frac{\hat{C}^{1/2}}{2^{p/2} C_{\overline{\delta}}^{p/2}}
\right)^{\vartheta/(p\zeta)}\delta^{\vartheta/(p\zeta\varsigma)}.
$$
The constants $C_T$ and $C_{\overline{\delta}}$ are defined below
and $\vartheta$ is as in (b).
First we derive an estimate for the norm of $\bu$ on $L^p(\Omega^-)$.
Consider the set
$$
S_1:=\left\{x\in\Omega^-:\:\mbox{dist}\,(x,\R^N\setminus\Omega^-)<{1}/{n_1}\right\},
$$
where $n_1$ is large enough so that 
$$
|S_1|\leq\textstyle\left(\frac{c_{2^*}^{\vartheta}\hat{\delta}}{3R^\vartheta}\right)^{2^*/(2^*-\vartheta)}.
$$
Using the H\"older inequality, in the first place we note
\begin{equation}\label{part1}
\int_{S_1}|\bu|^\vartheta\leq|\bu|_{2^*}^\vartheta|S_1|^{(2^*-\vartheta)/2^*}
\leq\frac{\|\bu\|^\vartheta}{c_{2^*}^{\vartheta}}
|S_1|^{(2^*-\vartheta)/2^*}\leq\frac{\hat{\delta}}{3}.
\end{equation}
Let $\beta>0$ be a constant such that $\am\geq\beta$ on $\Omega^-\setminus S_1$.
Consider now
$$
S_2:=\left\{x\in\Omega^-:\:|\bu(x)|\leq{\textstyle\left(\frac{\hat{\delta}}{3|\Omega|}\right)^{1/\vartheta}}\right\}.
$$
In the second place we note
\begin{equation}\label{part2}
\int_{S_2}|\bu|^\vartheta\leq\frac{\hat{\delta}}{3}.
\end{equation}
Let $c_{\hat{\delta}}>0$ be a constant such that 
$$
G(\xx u)\geq c_{\hat{\delta}}|u|^\vartheta,\qquad\mbox{for}\ x\in\Omega^-\setminus S_1\ 
\mbox{and}\ |u|\geq {\textstyle\left(\frac{\hat{\delta}}{3|\Omega|}\right)^{1/\vartheta}}.
$$
The existence of such a constant is implied by (b).
In the third place we note that
\begin{eqnarray*}
\iu(v)+1+\frac{\lambda}{2}\int\ap u^2+\int\ap F(\x u)&\geq&\mu\int_{\Omega^-\setminus(S_1\cup S_2)}\am G(\x u)\\
&\geq&\mu\beta c_{\hat{\delta}}
\int_{\Omega^-\setminus(S_1\cup S_2)}|\bu|^\vartheta,
\end{eqnarray*}
and so, for $\mu\geq\mu_\delta:=\textstyle\frac{3(\iu(v)+1+C)}{\beta c_{\hat{\delta}}\hat{\delta}}$,
where $C$ is such that $\frac{\lambda}{2}\int\ap u^2+\int\ap F(\x u)$ $\leq C$,
\begin{equation}\label{part3}
\int_{\Omega^-\setminus(S_1\cup S_2)}|\bu|^\vartheta\leq\frac{\hat{\delta}}{3}.
\end{equation}
Combining (\ref{part1}), (\ref{part2}) and (\ref{part3}),
$$
\int_{\Omega^-}|\bu|^\vartheta\leq\hat{\delta},
$$
for $\mu\geq\mu_\delta$. Interpolating the $L^p(\Omega)$ norm between
the $L^\vartheta(\Omega)$ and the $L^{2^*}(\Omega)$ norms,
\begin{eqnarray}
\int_{\Omega^-}|\bu|^p&\leq&\left(\int_{\Omega^-}|\bu|^\vartheta\right)^{p\zeta/\vartheta}
|\bu|_{2^*}^{p(1-\zeta)}\nonumber\\
&\leq&\hat{\delta}^{p\zeta/\vartheta}\frac{R^{p(1-\zeta)}}{c_{2^*}^{p(1-\zeta)}}\nonumber\\
&=&\frac{\hat{C}^{1/2}}
{2^{p/2}C_{\overline{\delta}}^{p/2}}\delta^{1/\varsigma}\leq\frac{\delta}{2}\label{part4}
\end{eqnarray}
for small $\delta$, and $\mu\geq\mu_\delta$.
Now we turn to the estimate for the norm of $\bu$ on $L^p(\Omega^+)$.
Let $q$ be the trace exponent $q=2(N-1)/(N-2)$.
If $N=1$ or 2 we take $q$ to be 
greater than 2.
There exists $C_{\overline{\delta}}$ such that
$$
\left(\int_{\partial\Omega^-}|\bu|^q\right)^{2/q}\leq
\overline{\delta}\int_{\Omega^-}|\nabla\bu|^2+
C_{\overline{\delta}}\left(\int_{\Omega^-}|\bu|^{p}\right)^{2/p}.
$$
This follows from \cite[bottom of p.\ 112]{A}.
From the expression for $\overline{\delta}$ and (\ref{part4}),
$$
\left(\int_{\partial\Omega^-}|\bu|^q\right)^{2/q}\leq
\hat{C}^{1/p}\delta^{2/(p\varsigma)},
$$
for $\mu\geq\mu_\delta$.
Using \cite[inequality (7.28) on p.\ 203]{LM},
\begin{eqnarray*}
\left(\int_{\Omega^+}|\bu|^2\right)^{1/2}&\leq&C\|\bu\|_{H^{-1/2}(\partial\Omega^+)}\\ &\leq&
C\|\bu\|_{L^{2}(\partial\Omega^+)}\\ &\leq& C_T\|\bu\|_{L^{q}(\partial\Omega^+)}\\
&=&
C_T\|\bu\|_{L^{q}(\partial\Omega^-)}\\
&\leq&C_T\hat{C}^{1/(2p)}\delta^{1/(p\varsigma)},
\end{eqnarray*}
for $\mu\geq\mu_\delta$.
This implies
\begin{equation}\label{part5}
\int_{\Omega^+}|\bu|^p
\leq\left(\int_{\Omega^+}|\bu|^2\right)^{p\varsigma/2}
|\bu|_{2^*}^{p(1-\varsigma)}
\leq C_T^{p\varsigma}\hat{C}^{\varsigma/2}\delta
{\textstyle\left(\frac{R}{c_{2^*}}\right)^{p(1-\varsigma)}}=
\frac{\delta}{2},
\end{equation}
for $\mu\geq\mu_\delta$.
Inequalities (\ref{part4}) and (\ref{part5}) together finally give
$$
|\bu|_p^p\leq\delta,
$$
for $\mu\geq\mu_\delta$.
\end{altproof}
\begin{altproof}{{\bf Claim~\ref{claim2}}}
Consider
$$ 
\varepsilon_1=\textstyle\left(\frac{\kappa_2}{4\sup\ap|\Omega|}\right)^{1/2}
$$ 
and $\str$ as in (\ref{SS3}). From (\ref{S3}),
$$
\int_{\{x\in\Omega^+:\:|w(x)|\geq\varepsilon_1\}\setminus \str}\ap w^2
\geq\frac{\kappa_2}{2}-\sup\ap\varepsilon_1^2|\Omega|
=\frac{\kappa_2}{4},
$$
for large $\mu$.
Let 
$$
M=\textstyle\left(\frac{8R^{2^*}\sup\ap}{\kappa_2 c_{2^*}^{2^*}}\right)^{1/(2^*-2)}.
$$
Since, by Chebyshev's inequality,
$$
\left|\{x\in\Omega^+:\:|w(x)|>M\}\right|\leq\frac{|w|_{2^*}^{2^*}}{M^{2^*}}\leq
\frac{R^{2^*}}{c_{2^*}^{2^*}M^{2^*}},
$$
we have
\begin{eqnarray*}
\int_{\{x\in\Omega^+:\:|w(x)|>M\}}\ap w^2&\leq&\sup\ap
|w|_{2^*}^2\left|\{x\in\Omega^+:\:|w(x)|>M\}\right|^{(2^*-2)/2^*}\\
&\leq&\sup\ap\frac{R^{2^*}}{c_{2^*}^{2^*}M^{2^*-2}}\\
&=&\frac{\kappa_2}{8}.
\end{eqnarray*}
Choosing $\str$ as in (\ref{SS3}), 
$$
\int_{\{x\in\Omega^+:\:\varepsilon_1\leq|w(x)|\leq M\}\setminus \str}\ap w^2\geq\frac{\kappa_2}{4}-\frac{\kappa_2}{8}=
\frac{\kappa_2}{8},
$$
for large $\mu$.
Let $c_f^\prime>0$ be such that 
$f^\prime(\xx w)-\frac{f(\xx w)}{w}\geq c_f^\prime$ for $|w|\in[\eta\varepsilon_1,M/\eta]$
and $x\in\Omega^+\setminus \str$.
Then $\g'(t;w)\geq\eta c_f^\prime\kappa_2/8$ for $t\in[\eta,1/\eta]$ and $\mu$ large.
This proves Claim~\ref{claim2}.
\end{altproof}


\begin{thebibliography}{99}

\bibitem{AW}
Ackermann, N.; Weth, T..
Multibump solutions of nonlinear periodic Schrödinger equations in a degenerate setting.
Commun.\ Contemp.\ Math.\ 7 (2005), no.\ 3, 269--298.


\bibitem{A}
Adams, R.A..
Sobolev spaces.
Pure and Applied Mathematics, Vol. 65. Academic Press, New York-London, 1975.


\bibitem{AdP}
Alama, S.; del Pino, M..
Solutions of elliptic equations with indefinite nonlinearities via Morse theory and linking.
Ann.\ Inst.\ H.\ Poincaré Anal.\ Non Linéaire 13 (1996), no.\ 1, 95--115.


\bibitem{AT}
Alama, S.; Tarantello, G..
On semilinear elliptic equations with indefinite nonlinearities.
Calc.\ Var.\ Partial Differential Equations 1 (1993), no.\ 4, 439--475.


\bibitem{BCW}
Bartsch, T.; Clapp, M.; Weth, T..
Configuration spaces, transfer, and 2-nodal solutions of a semiclassical nonlinear
Schr\"{o}dinger equation.
Math.\ Ann.\ 338 (2007), no.\ 1, 147--185 

\bibitem{BW}
Bartsch, T.; Weth, T..
A note on additional properties of sign changing solutions to superlinear elliptic equations.
Topol.\ Methods Nonlinear Anal. 22 (2003), no.\ 1, 1--14.

\bibitem{BGH}
Bonheure, D.; Gomes, J.M.; Habets, P..
Multiple positive solutions of superlinear elliptic problems with sign-changing weight.
J.\ Differential Equations 214 (2005), no.\ 1, 36--64.


\bibitem{CCN}
Castro, A.; Cossio, J.; Neuberger, J.M..
A sign-changing solution for a superlinear Dirichlet problem.
Rocky Mountain J.\ Math.\ 27 (1997), no.\ 4, 1041--1053.


\bibitem{CZR1}
Coti Zelati, V.; Rabinowitz, P.H..
Homoclinic orbits for second order Hamiltonian systems possessing superquadratic potentials.
J.\ Amer.\ Math.\ Soc.\ 4 (1991), no.\ 4, 693--727.


\bibitem{CZR2}
Coti Zelati, V.; Rabinowitz, P.H..
Homoclinic type solutions for a semilinear elliptic PDE on $\R^n$.
Comm.\ Pure Appl.\ Math.\ 45 (1992), no.\ 10, 1217--1269.


\bibitem{dPF}
del Pino, M.; Felmer, P..
Multi-peak bound states for nonlinear Schrödinger equations.
Ann.\ Inst.\ H.\ Poincaré Anal.\ Non Linéaire 15 (1998), no.\ 2, 127--149.


\bibitem{dPF2}
del Pino, M.; Felmer, P..
Semi-classical states of nonlinear Schrödinger equations: a variational reduction method.
Math.\ Ann.\ 324 (2002), no.\ 1, 1--32.



\bibitem{GG}
Girão, P.M.; Gomes, J.M..
Multibump nodal solutions for an indefinite superlinear elliptic problem.
To appear in J.\ Differential Equations.

\bibitem{LW}
Li, Y.; Wang, Z.Q..
Gluing approximate solutions of minimum type on the Nehari manifold.
Proceedings of the USA-Chile Workshop on Nonlinear Analysis (Viña del Mar-Valparaiso, 2000), 215--223,
Electron.\ J.\ Differ.\ Equ.\ Conf., 6, Southwest Texas State Univ., San Marcos, TX, 2001.


\bibitem{LM}
Lions, J.-L.; Magenes, E..
Problèmes aux limites non homogènes et applications. Vol.\ 1.
Travaux et Recherches Mathématiques, No.\ 17 Dunod, Paris, 1968

\bibitem{MP}
Musso, M.; Passaseo, D..
Multibump solutions for a class of nonlinear elliptic problems. 
Calc.\ Var.\ Partial Differential Equations 7 (1998), no.\ 1, 53--86. 

\bibitem{RT}
Ramos, M.; Tavares, H..
Solutions with multiple spike patterns for an elliptic system. 
Calc.\ Var.\ Partial Differential Equations 31 (2008), no.\ 1, 1--25.

\bibitem{S}
Séré, É..
Existence of infinitely many homoclinic orbits in Hamiltonian systems.
Math.\ Z.\ 209 (1992), no.\ 1, 27--42.

\end{thebibliography}
\end{document}